\documentclass[11pt]{amsart}

\numberwithin{equation}{section}
\usepackage{mathrsfs,txfonts}
\usepackage[dvips]{graphicx}
\usepackage{latexsym,bm}
\usepackage{amsfonts}
\usepackage{indentfirst}
\usepackage{amsthm}
\usepackage{cases}
\usepackage{pifont}
\usepackage{amsmath}
\usepackage{amssymb}
\usepackage[misc]{ifsym}
\allowdisplaybreaks
\newtheorem{lem}{\quad\textbf{\Large Lemma}}[section]
\newtheorem{thm}[lem]{\quad\textbf{\Large Theorem}}

\def\squarebox#1{\hbox to #1{\hfill\vbox to #1{\vfill}}}

\usepackage{enumerate}
\begin{document}
\title[Semi-finite forms of the quintuple product identity]
{A new semi-finite form of the quintuple product identity }
\author{Jun-Ming Zhu%$^{\textrm{\Letter}}$
}
%\author{Zi-Qian Huang}
\address{Department of Mathematics, Luoyang Normal University,
Luoyang City, Henan Province 471934, China}
\email{junming\_zhu@163.com}
%
%\address{Department of Physics, Guangxi University of Chinese
%Medicine, Nanning City,  Guangxi Province 530200, China}
%\email{huangziqian2004@163.com}

\thanks{Keywords: quintuple product identity; semi-finite form; q-series.
\\\indent MSC (2010):   33D15, 11F27.
 \\\indent This research is supported by the
 Natural Science Foundation of China
 (Grant No. 11871258).
}

\begin{abstract}
The quintuple product identity are deduced from a new semi-finite form, which are obtained
from the very-well-poised $_6\phi_5$ series.
% As its application, an identity on $q$-gamma function is obtained.
\end{abstract}
 \maketitle
%%%%%%%%%%%%%%%%%%%%%%%%%%%%%%%%%%%%%%%%%%%%%%%%%%%%%%%%%%%%%%%%%%%%%%%%%%%%%

\section{Introduction}

Throughout this paper the product of $q$-shifted factorials are defined by
\begin{equation*}\label{def}
(a;q)_\infty=\prod_{l=0}^{\infty}(1-aq^{l}) \qquad \mbox{and}\qquad
(a;q)_n=\frac{(a;q)_\infty}{(aq^n;q)_\infty},
\end{equation*}
for $n\in{\mathbb{Z}}$ and  $|q|<1$, with the following abbreviated
multiple parameter notation
$$
(a,b,\cdots,c; q)_k=(a;q)_k(b;q)_k\cdots(c; q)_k,\quad
k\in{\mathbb{Z\cup\{\infty\}}}.
$$

In a recent short paper \cite{zh1}, the author and Zhang give the theorem below.
\begin{thm}\label{1-qunn}
{( The first semi-finite form of the quintuple product identity )}
 There holds
\begin{equation}\label{s1}
\sum_{k=0}^\infty(1+zq^k)\frac{(z^2;q)_k q^{k^2}z^k}{(q;q)_k }
=(-z;q)_\infty(z^2q;q^2)_\infty,\quad z\in{\mathbb{C}}.
\end{equation}
\end{thm}
This is proved to be a semi-finite form of
the celebrated quintuple product identity, which can be stated as
\begin{equation}\label{ax}
\sum_{k=-\infty}^\infty(-1)^kq^{k(3k-1)\over2}z^{3k}(1+zq^k)
=\frac{(
 q,q/z^2,z^2;q)_\infty}{(z,q/z
;q)_\infty},\quad z\neq0.
\end{equation}
The  identity \eqref{ax}
 or its equivalent forms can be found in Fricke
\cite{fr}, Ramanujan \cite{ra} and Watson \cite{wa}.
This is an important  identity in combinatorics, number~theory
and special functions. For the historical note and various proofs, we
refer the reader to the paper \cite{co}.

 In this short note, we give a new semi-finite
form of the quintuple product identity
from the very-well-poised $_6\phi_5$ series, and then, the quintuple product identity  \eqref{ax} is deduced.
\begin{thm}\label{2-qunn}
{(The second semi-finite form of the quintuple product identity)}
 There holds
\begin{equation}\label{2qun}
\sum_{k=0}^\infty(1-z^2q^{2k+1})\frac{(z^2q;q)_k q^{k^2}
z^{k}}{(q;q)_k}\! =(-zq;q)_\infty(z^2q;q^2)_\infty,\quad
z\in{\mathbb{C}}.
\end{equation}
\end{thm}
For more semi-finite forms of the  bilateral basic
hypergeometric series, the reader  can consult the papers
\cite{cf,jo,zh}.

\section{The proofs of Theorem \ref{2-qunn} and the quintuple product identity \eqref{ax} }
\begin{proof}[Proof of Theorem \ref{2-qunn}]We begin with the  non-terminating very-well poised $_6\phi_5$ series \cite[ p. 356, Eq. (II.
20)]{gr}:
\begin{equation}\label{ds}
\sum_{k=0}^\infty\frac{(a,qa^{1\over2},-qa^{1\over2},b,c,d;q)_k}{(q,a^{1\over2},-a^{1\over2},aq/b,aq/c,aq/d;q)_k}\left(\frac{qa}{bcd}\right)^k
=\frac{(aq,aq/bc,aq/bd,aq/cd;q)_\infty}
 {(aq/b,aq/c,aq/d,aq/bcd;q)_\infty},
\end{equation}
with $\left|\frac{qa}{bcd}\right|<1$. Replacing $(a^{1\over2},b)$ by
$(zq^{1\over2},zq)$ and letting $c=d\rightarrow\infty$, in the
resulting identity, we multiply both sides by $1-z^2q$, to give
\eqref{2qun}.
\end{proof}

Now we deduce the quintuple product identity \eqref{ax} from Theorem
\ref{2-qunn}. For a nonnegative integer $n$ and $z\neq0$, we replace
$z$ by $zq^{-n}$ in \eqref{2qun} to have
\begin{align}
&\sum_{k=0}^{\infty}\frac{(1-z^2q^{2k-2n+1})(z^2q^{1-2n};q)_k}{(q;q)_k}z^{k} q^{k^2-kn}\notag\\
&=\sum_{k=-n}^{\infty}\frac{(1-z^2q^{2k+1})(z^2q^{1-2n};q)_{k+n}}{(q;q)_{k+n}}z^{k+n} q^{k^2+kn}\notag\\
&=\frac{z^n(z^2q^{1-2n};q)_{n}}{(q;q)_{n}}
   \sum_{k=-n}^{\infty}\frac{(-1)^k(1-z^2q^{2k+1})(q^{n-k}/z^2;q)_{k}}{(q^{n+1};q)_k}z^{3k} q^{3k^2+k\over2}\label{mi1}\\
&=(-zq^{1-n};q)_\infty(z^2q^{1-2n};q^2)_\infty=\frac{(z^2q^{1-2n};q)_\infty}{(zq^{1-n};q)_\infty}\notag\\
&=\frac{z^n(z^2q^{1-2n};q)_{n}(1/z^2;q)_{n}(z^2q;q)_\infty}{(1/z;q)_{n}(zq;q)_\infty}.\label{mi2}
\end{align}
From \eqref{mi1} and \eqref{mi2}, we have
\begin{align}\label{qua}
\sum_{k=-n}^{\infty}\frac{(-1)^k(1-z^2q^{2k+1})(q^{n-k}/z^2;q)_{k}}{(q^{n+1};q)_k}z^{3k}
q^{3k^2+k\over2}=\frac{(q,1/z^2;q)_{n}(z^2q;q)_\infty}{(1/z;q)_{n}(zq;q)_\infty}.
\end{align}
Then, restricting $z$ in any compact subset of $0<|z|<\infty$ and
letting $n\rightarrow\infty$ gives
\begin{align}\label{qua1}
\sum_{k=-\infty}^{\infty}(-1)^k(1-z^2q^{2k+1})z^{3k}
q^{3k^2+k\over2}=\frac{(q,1/z^2,z^2q;q)_\infty}{(1/z,zq;q)_\infty}.
\end{align}
By analytic
continuation, the restriction on $z$ may be relaxed. This is an equivalent form of the identity \eqref{ax}.

It is easy to verify that the identity \eqref{qua} is   equivalent
to \eqref{2qun}.

Note that the $_6\phi_5$ formula \eqref{ds} also implies the
identity in Theorem \ref{1-qunn}.

%\section{Combinatorial interpretations of Theorem \ref{1-qunn} and
%\ref{2-qunn}}

%{\bf Acknowledgment.} The author wishes to thank the editor and the
%referee for their valuable comments and advices. The authors are
%partially supported by the
% National Science Foundation of China
%(Grant No. 11871258).

%

\end{document}